\documentclass[oneside,english]{amsart}
\usepackage[T1]{fontenc}
\usepackage[latin9]{inputenc}
\usepackage{mathrsfs}
\usepackage{amstext}
\usepackage{amsthm}
\usepackage{amssymb}
\usepackage[all]{xy}

\makeatletter
\numberwithin{equation}{section}
\numberwithin{figure}{section}
\theoremstyle{plain}
\newtheorem{thm}{\protect\theoremname}
\theoremstyle{remark}
\newtheorem{rem}[thm]{\protect\remarkname}
\theoremstyle{definition}
\newtheorem{defn}[thm]{\protect\definitionname}
\theoremstyle{definition}
\newtheorem{example}[thm]{\protect\examplename}
\theoremstyle{plain}
\newtheorem{cor}[thm]{\protect\corollaryname}

\usepackage[dvipsnames,svgnames,x11names,hyperref]{xcolor}
\usepackage{lmodern}
\renewcommand*{\epsilon}{\varepsilon}
\usepackage{amssymb}
\usepackage{amsfonts}
\usepackage{egothic}
\usepackage{xcolor}
\usepackage[T1]{fontenc}

\usepackage{url}
\usepackage{amsthm}
\usepackage{aliascnt}
\usepackage{lipsum}
\usepackage{mathrsfs}
\usepackage{esint}
\usepackage{url}
\usepackage{amsmath}
\usepackage{dsfont}
\usepackage{bbm}
\usepackage{pdflscape}
\usepackage{bbm}
\usepackage{bussproofs}

\makeatletter
\def\matrixobject@{%
  \edef \next@{={\DirectionfromtheDirection@ }}%
  \expandafter \toks@ \next@ \plainxy@
  \let\xy@@ix@=\xyq@@toksix@
  \xyFN@ \OBJECT@}
\let\xy@entry@@norm=\entry@@norm
\def\entry@@norm@patched{%
  \let\object@=\matrixobject@
  \xy@entry@@norm }
\AtBeginDocument{\let\entry@@norm\entry@@norm@patched}
\makeatother

\newcommand{\twocong}[2][0.5]{\ar@{}[#2] \save ?(#1)*{\cong}\restore}
\newcommand{\twoeq}[2][0.5]{\ar@{}[#2] \save ?(#1)*{=}\restore}
\newcommand{\ltwocell}[3][0.5]{\ar@{}[#2] \ar@{=>}?(#1)+/r 0.15cm/;?(#1)+/l 0.15cm/^{#3}}
\newcommand{\rtwocell}[3][0.5]{\ar@{}[#2] \ar@{=>}?(#1)+/l 0.15cm/;?(#1)+/r 0.15cm/^{#3}}
\newcommand{\utwocell}[3][0.5]{\ar@{}[#2] \ar@{=>}?(#1)+/d  0.15cm/;?(#1)+/u 0.15cm/_{#3}}
\newcommand{\dtwocell}[3][0.5]{ \ar@{}[#2] { \ar@{=>}?(#1)+/u  0.15cm/;?(#1)+/d 0.15cm/^{#3}}}
\newcommand{\ultwocell}[3][0.5]{\ar@{}[#2] \ar@{=>}?(#1)+/dr  0.15cm/;?(#1)+/ul 0.15cm/^{#3}}
\newcommand{\urtwocell}[3][0.5]{\ar@{}[#2] \ar@{=>}?(#1)+/dl  0.15cm/;?(#1)+/ur 0.15cm/^{#3}}
\newcommand{\dltwocell}[3][0.5]{\ar@{}[#2] \ar@{=>}?(#1)+/ur  0.15cm/;?(#1)+/dl 0.15cm/^{#3}}
\newcommand{\drtwocell}[3][0.5]{\ar@{}[#2] \ar@{=>}?(#1)+/ul  0.15cm/;?(#1)+/dr 0.15cm/^{#3}}

\newcommand{\myar}[2]{\ar^-{#1}[#2]}

\usepackage[colorlinks=true,linkcolor=Dark Red, citecolor=Dark Green,linktoc=all]{hyperref}
\usepackage{amsmath}
\usepackage{cleveref}

\makeatletter

  \def\make@df@tag@@#1{%
    \gdef\df@tag{%
      \maketag@@@{\Hy@make@anchor#1}%
      \def\@currentlabel{#1}%
      \def\cref@currentlabel{[equation][2147483647][]#1}%
    }%
  }
  \def\make@df@tag@@@#1{%
    \gdef\df@tag{%
      \tagform@{\Hy@make@anchor#1}%
      \toks@\@xp{\p@equation{#1}}%
      \edef\@currentlabel{\the\toks@}%
      \edef\cref@currentlabel{[equation][2147483647][]\the\toks@}
    }%
  }
\makeatother

\numberwithin{thm}{subsection}

\makeatother

\usepackage{babel}
\providecommand{\corollaryname}{Corollary}
\providecommand{\definitionname}{Definition}
\providecommand{\examplename}{Example}
\providecommand{\remarkname}{Remark}
\providecommand{\theoremname}{Theorem}

\begin{document}
\title{Generalized Nerves of Monads}
\author{Charles Walker}
\address{Department of Mathematics and Statistics, Masaryk University, Kotl{\'a}{\v r}sk{\'a}
2, Brno 61137, Czech Republic }
\email{\tt{walker@math.muni.cz}}
\thanks{This work was supported by the Operational Programme Research, Development
and Education Project \textquotedblleft Postdoc@MUNI\textquotedblright{}
(No. CZ.02.2.69/0.0/0.0/18\_053/0016952)}
\begin{abstract}
One interpretation of the Kleisli construction (given by Miranda and
related to work of Paré) is as a nerve sending a 2-monad $P$ to the
Kleisli double category of $P$. In this paper we find more general
nerve constructions on the 2-categories of monads, which also give
fully faithful nerve 2-functors $\mathbf{Mnd}\left(\mathcal{K}\right)\to\mathbf{Dbl}$. 
\end{abstract}

\maketitle
\tableofcontents{}

\section{Introduction}

Double categories generalize categories by allowing for two different
types of morphisms, called horizontal and vertical morphisms, and
consist of these two types of morphism as well as squares 
\[
\xymatrix@=1.5em{X\ar[rr]\ar@{~>}[dd] &  & Y\ar@{~>}[dd]\\
\\
X'\ar[rr]\urtwocell{urur}{\phi} &  & Y'
}
\]
satisfying some compositional axioms. One of the most the natural
examples of a double category, is the Kleisli double category of a
monad $P$. These double categories have been studied by Paré \cite{PareKlTalk,AdjDbl},
and for a given monad $P$ on a category $\mathcal{C}$ are the double
categories with horizontal morphisms those of $\mathcal{C}$ and vertical
morphisms $X\rightsquigarrow Y$ are morphisms $X\to PY$ in $\mathcal{C}$
(called Kleisli arrows)\footnote{Of course one can switch the horizontal and vertical morphisms, giving
the transpose.}. A square $\phi$ as on the left below

\[
\xymatrix@=1.5em{X\ar@{~>}[dd]_{f}\ar[rr]^{a} &  & X'\ar@{~>}[dd]^{g} &  &  &  & X\ar[dd]_{f}\ar[rr]^{a} &  & X'\ar[dd]^{g}\\
\\
Y\ar[rr]_{b}\urtwocell{rruu}{\phi} &  & Y' &  &  &  & PY\ar[rr]_{Pb} &  & PY'
}
\]
is then the condition that the right square above commutes. One view
of this construction given by Miranda \cite{Miranda} is as the nerve 

\[
\xymatrix@=1.5em{\mathbf{Kl}\left(\mathbf{Cat}\right)\ar[rr]^{N} &  & \left[\Delta^{\textnormal{op}},\mathbf{Cat}\right]\\
 & \Delta\ar[ur]\ar[ul]
}
\]
where $\mathbf{Kl}\left(\mathbf{Cat}\right)$ is the completion of
$\mathbf{Cat}$ under Kleisli objects as defined in \cite{StreetFTM2}.
Here the nerve factors through double categories and so may be written
\[
N\colon\mathbf{Kl}\left(\mathbf{Cat}\right)\to\mathbf{Dbl}
\]
If one switches the vertical and horizontal transformations between
double functors (equivalent to considering the transpose) one recovers
the embedding 
\[
N^{T}\colon\mathbf{Mnd}\left(\mathbf{Cat}\right)\to\mathbf{Dbl}
\]
 from the 2-category of monads $\mathbf{Mnd}\left(\mathbf{Cat}\right)$
\cite{StreetFTM2}\footnote{In the notation of \cite{StreetFTM2} this is $\mathbf{Mnd}\left(\mathbf{Cat}^{\textnormal{op}}\right)^{\textnormal{op}}$.}.

The goal of this paper is to find more general examples of fully faithful
nerves, and thus more examples of embeddings $\mathbf{Mnd}\left(\mathbf{Cat}\right)\to\mathbf{Dbl}$.
To do this, we replace the Kleisli completion $\mathbf{Kl}\left(\mathbf{Cat}\right)$
by a 2-category $\mathbf{mKl}\left(\mathbf{Cat}\right)$ which is
like $\mathbf{Kl}\left(\mathbf{Cat}\right)$ in that we have a comparison
2-functor
\[
\Phi\colon\mathbf{mKl}\left(\mathbf{Cat}\right)\to\mathbf{Kl}\left(\mathbf{Cat}\right)
\]
satisfying certain axioms. This $\Phi$ should be identity of objects
and 1-cells (and is thus specified by its action on 2-cells), and
should satisfy a number of axioms. Of these axioms, one is non-trivial,
which is that each $N\colon\mathbf{mKl}\left(\mathbf{Cat}\right)\to\mathbf{Dbl}$,
should send a monad $P$ to a double category containing all squares
of the form
\[
\xymatrix@=1.5em{X\ar@{~>}[dd]_{f}\ar[rr]^{\Phi f} &  & PY\ar@{~>}[dd]^{\epsilon_{Y}}\\
\\
Y\ar@{=}[rr] &  & Y
}
\]
for an $\epsilon_{Y}$ in $\mathbf{mKl}\left(\mathbf{Cat}\right)$
where $\Phi\epsilon_{Y}$ is the identity on $PY$ (i.e. the counit
of the Kleisli adjunction). To make this clear we briefly give some
examples:
\begin{itemize}
\item With $\Phi$ the identity, and so $\mathbf{mKl}\left(\mathbf{Cat}\right)=\mathbf{Kl}\left(\mathbf{Cat}\right)$,
it is easy to check 
\[
\xymatrix@=1.5em{X\ar@{~>}[dd]_{f}\ar[rr]^{f} &  & PY\ar@{~>}[dd]^{\epsilon_{Y}}\\
\\
Y\ar@{=}[rr] &  & Y
}
\]
with $\epsilon_{Y}=\textnormal{id}_{PY}$ is a square in the Kleisli
double category.
\item Instead of a 2-cell $X\rightsquigarrow Y$ being a Kleisli map $X\to PY$.
Take $\mathbf{mKl}\left(\mathbf{Cat}\right)$ to have 2-cells being
pairs of maps $Y\to X\to PY$ composing to the unit $\eta_{Y}$. One
may regard this as a Kleisli map with some extra structure. Here $\Phi$
simply forgets the extra structure, recovering the underlying Kleisli
map, and each $\epsilon_{Y}$ is the factorization of the unit as
$\textnormal{id}_{PY}\cdot\eta_{Y}\colon Y\to PY\to PY$. For a general
2-cell $\tau\cdot\pi\colon Y\to X\to PY$, the left square below lies
in $N\left(P\right)$ since it is a square on the Kleisli arrow component
$\tau$ (as earlier) and on the extra structure $\pi$ (by the square
on the right below).
\[
\xymatrix@=1.5em{X\ar@{~>}[dd]_{f}\ar[rr]^{f} &  & PY\ar@{~>}[dd]^{\epsilon_{Y}} &  &  & X\ar[rr]^{\tau} &  & PY\\
\\
Y\ar@{=}[rr] &  & Y &  &  & Y\ar@{=}[rr]\ar[uu]^{\pi} &  & Y\ar[uu]_{\eta_{Y}}
}
\]
As this example satisfies the properties we require, it will give
a transpose nerve functor
\[
N^{T}\colon\mathbf{Mnd}\left(\mathbf{Cat}\right)\to\mathbf{Dbl}
\]
which the reader should recognize the embedding
\[
\mathbf{Mnd}_{\times}\left(\mathbf{Cat}\right)\to\mathbf{AWFS}\to\mathbf{Dbl}
\]
given by Bourke-Garner \cite{AWFS2} in the context of algebraic weak
factorization systems (AWFS), sending a monad $P$ to the double category
of $P$-embeddings (also called $P$ split-monos)\footnote{In order to factor through $\mathbf{AWFS}$ one uses monads which
are suitably compatible with products.}.
\item One different example, though analogous in some ways to the previous
example, is to take pairs of Kleisli maps $e\colon Y\to PX$ and $s\colon X\to PY$
such that
\[
\xymatrix@=1.5em{X\ar[r]^{s}\ar[rd]_{\eta_{X}} & PY\ar[r]^{Pe} & P^{2}X\\
 & PX\ar[ur]_{\eta_{PX}}
}
\]
which we may think of a strong version of Kleisli split-epis (hence
the notation $s$ and $e$) as composition by $\mu_{X}\colon P^{2}X\to PX$
recovers the Kleisli identity arrow $\eta_{X}\colon X\to PX$. Here
an $\epsilon_{Y}$ is the pair $\eta_{PY}\cdot\eta_{Y}\colon Y\to P^{2}Y$,
$\textnormal{id}_{PY}\colon PY\to PY$, with the extra structure $e$
still giving a square since
\[
\xymatrix@=1.5em{PY\ar[rr]^{Pe} &  & P^{2}X\\
\\
X\ar@{=}[rr]\ar[uu]^{s} &  & X\ar[uu]_{\eta_{PX}\cdot\eta_{X}}
}
\]
ensuring the required conditions also hold on the non-Kleisli components.
\end{itemize}
Lastly, we mention iterated versions of these constructions. One can
iterate the Kleisli nerve to recover embeddings
\[
\mathbf{Mnd}^{n}\left(\mathbf{Cat}\right)\to\left(n+1\right)\mathbf{Fold}
\]
into $\left(n+1\right)$-fold categories \cite{Miranda}, thus sending
distributive laws to triple categories and so on. We give a more general
version of this, allowing for iteration of any of the examples of
nerves in any order.

The original motivation for this work was to use these nerves to give
simpler proofs about various results on distributive laws, however
this will now be left for a future paper.

\section{Background}

In this section we will recall the neccessary background knowledge
on internal categories, double categories, 2-categories of monads,
and completions under Kleisli objects.

\subsection{2-categories of monads in $\mathcal{K}$}

We first recall the definition of the 2-category of monads. We warn
the reader that the morphisms are defined differently to \cite{StreetFTM},
with the direction of each monad morphism's 2-cell component reversed.
\begin{rem}
In the notation of \cite{StreetFTM} the following 2-category was
denoted $\mathbf{Mnd}\left(\mathcal{K}^{\textnormal{op}}\right)^{\textnormal{op}}$,
though we will avoid the 'op's in this paper for brevity. The choice
of the direction of the 2-cell $\xi$ we are using here works most
nicely with Kleisli constructions, which is unsurprising as $\xi\colon FP\to QF$
gives a Kleisli arrow $FP\rightsquigarrow F$.
\end{rem}

\begin{defn}
Given a 2-category the $\mathcal{K}$, the 2-category $\mathbf{Mnd}\left(\mathcal{K}\right)$
of monads in $\mathcal{K}$ has:

\emph{objects:} monads which consist of data $\left(P\colon\mathcal{C}\to\mathcal{C},\eta\colon1\Rightarrow P,\mu\colon P^{2}\Rightarrow P\right)$
in $\mathcal{K}$ rendering commutative
\[
\xymatrix{P\ar[rr]^{\eta P}\ar[rrd]_{\textnormal{id}} &  & PP\ar[d]_{\mu} &  & P\ar[ll]_{P\eta}\ar[ldl]^{\textnormal{id}} &  & PPP\ar[rr]^{\mu P}\ar[d]_{P\mu} &  & PP\ar[d]^{\mu}\\
 &  & P &  &  &  & PP\ar[rr]_{\mu} &  & P
}
\]

\emph{morphisms:} a morphism $P\to Q$ is a $\xi$ as on the left
below satisfying the triangle and pentagon axioms
\[
\xymatrix@=1.5em{\mathcal{C}\ar[d]_{P}\ar[rr]^{F} &  & \mathcal{D}\ar[d]^{Q} & FP\ar[rr]^{\xi} &  & QF & FP^{2}\ar[r]^{\xi P}\ar[d]_{F\mu} & QFP\ar[r]^{Q\xi} & Q^{2}P\ar[d]^{\mu P}\\
\mathcal{C}\ar[rr]_{F}\urtwocell{urr}{\xi} &  & \mathcal{D} &  & F\ar[ur]_{\eta_{Q}F}\ar[ul]^{F\eta_{P}} &  & FP\ar[rr]_{\xi} &  & QP
}
\]

\emph{2-cells:} a 2-cell $\left(F,\xi\right)\Rightarrow\left(F',\xi'\right)\colon P\to Q$
is a map $\alpha\colon F\to F'$ such that 
\[
\xymatrix@=1.5em{FP\ar[r]^{\xi}\ar[d]_{\alpha P} & QF\ar[d]^{Q\alpha}\\
F'P\ar[r]^{\xi'} & QF'
}
\]
\end{defn}

\begin{rem}
The inclusion $\mathbf{inc}\colon\mathcal{\mathbf{Cat}}\to\mathbf{Mnd}\left(\mathbf{Cat}\right)$
is the middle of an adjoint triple $\mathbf{Kl}\dashv\mathbf{inc}\dashv\mathbf{und}$\footnote{Note in \cite{StreetFTM} the direction of 2-cells is chosen differently,
thus giving a slightly different triple involving the Eilenberg-Moore
construction.}. Here $\mathbf{und}$ is the underlying functor which takes a monad
$P$ on a category $\mathcal{C}$ to $\mathcal{C}$, and $\mathbf{Kl}$
takes a monad $P$ to the Kleisli category of $P$ denoted $\mathcal{C}_{P\textnormal{-kl}}$
or just $\mathcal{C}_{P}$. Though trivial, we will often make use
of the underlying-inclusion adjunction, which simply says that monad
morphisms $\left(\textnormal{id},\mathcal{C}\right)\to\left(Q,\mathcal{D}\right)$
out of an identity monad (which must be as below)
\[
\xymatrix@=1.5em{\mathcal{C}\ar[d]_{\textnormal{id}}\ar[rrrr]^{F} &  &  &  & \mathcal{D}\ar[d]^{Q}\\
\mathcal{C}\ar[rrrr]_{F}\urtwocell{urrrr}{\xi=\eta_{Q}F} &  &  &  & \mathcal{D}
}
\]
correspond to morphisms $F:\mathcal{C}\to\mathcal{D}$. 
\end{rem}

\subsection{The Kleisli-object completion of a 2-category $\mathcal{K}$}

Since $\mathbf{Kl}$ is a left adjoint, its existence may be defined
by a representability condition. That is to give for each monad $P$
a representing object $\mathbf{Kl}\left(P\right)$ for the copresheaf
\[
\mathcal{K}\to\mathbf{Cat}\colon\mathcal{A}\mapsto\mathbf{Mnd}\left(\mathcal{K}\right)\left(P,\textnormal{id}_{\mathcal{A}}\right)
\]
The category $\mathbf{Mnd}\left(\mathcal{K}\right)\left(P,\textnormal{id}_{\mathcal{A}}\right)$
is simply the category of $P$-opalgebras on $\mathcal{A}$, and so
$\mathbf{Kl}\left(P\right)$ is the universal $P$-opalgebra. The
Kleisli-object completion $\mathbf{Kl}\left(\mathcal{K}\right)$ of
a 2-category $\mathcal{K}$ freely completes a 2-category $\mathcal{K}$
under these Kleisli objects (universal opalgebras) \cite{StreetFTM2}.
\begin{defn}
\cite{StreetFTM2} Given a 2-category the $\mathcal{K}$, the 2-category
$\mathbf{Kl}\left(\mathcal{K}\right)$ called the Kleisli-object completion
of $\mathcal{K}$ has the same objects and 1-cells of $\mathbf{Mnd}\left(\mathcal{K}\right)$,
but a 2-cell $\left(F,\xi\right)\rightsquigarrow\left(F',\xi'\right)\colon P\to Q$
is now a map $\alpha\colon F\to QF'$ such that
\begin{equation}
\xymatrix@=1.5em{FP\ar[r]^{\xi}\ar[d]_{\alpha P} & QF\ar[r]^{Q\alpha} & Q^{2}F'\ar[d]^{\mu F'}\\
QF'P\ar[r]_{Q\xi'} & Q^{2}F'\ar[r]_{\mu F'} & QF'
}
\label{klcond}
\end{equation}
The horizontal composite of a pair of 2-cells $\alpha\colon\left(F,\xi\right)\rightsquigarrow\left(F',\xi'\right)\colon P\to Q$
and $\beta\colon\left(G,\phi\right)\rightsquigarrow\left(G',\phi'\right)\colon Q\to R$
is the map
\[
\xymatrix@=1.5em{GF\ar[r]^{G\alpha} & GQF'\ar[r]^{\phi F'} & RGF'\ar[r]^{R\beta F'} & RRG'F'\ar[rr]^{\mu G'F'} &  & RG'F'}
\]
\end{defn}

\begin{example}
Disregarding that $\mathbf{Span}\left(\mathbf{Set}\right)$ is only
a bicategory, the 2-category $\mathbf{Cat}$ is equivalent to $\mathbf{Kl}_{*}\left(\mathbf{Span}\left(\mathbf{Set}\right)\right)$
where $*$ denotes that we only take monad morphisms whose underlying
1-cell is a left adjoint span.
\end{example}

\subsection{Kleisli double categories}

We now recall the basic definitions of internal categories and double
categories, so that we may recall the notion of Kleisli double categories
and how to see them as a nerve.
\begin{defn}
A \emph{category internal to a category} $\mathcal{E}$ with pullbacks
consists of objects $\mathcal{C}_{0}$ and $\mathcal{C}_{1}$, source
and target maps $s,t\colon\mathcal{C}_{1}\to\mathcal{C}_{0}$ a unit
map $i\colon\mathcal{C}_{0}\to\mathcal{C}_{1}$ and a composition
map $\circ\colon\mathcal{C}_{1}\times_{\mathcal{C}_{0}}\mathcal{C}_{1}\to\mathcal{C}_{1}$
satisfying 
\[
\xymatrix@=1.5em{\mathcal{C}_{0}\ar[rr]^{i}\ar@{=}[rdrd] &  & \mathcal{C}_{1}\ar[dd]^{s} & \mathcal{C}_{0}\ar[rr]^{i}\ar@{=}[rdrd] &  & \mathcal{C}_{1}\ar[dd]^{t} & \mathcal{C}_{1}\times_{\mathcal{C}_{0}}\mathcal{C}_{1}\ar[rr]^{\circ}\ar[dd]_{\pi_{1}} &  & \mathcal{C}_{1}\ar[dd]^{s} & \mathcal{C}_{1}\times_{\mathcal{C}_{0}}\mathcal{C}_{1}\ar[rr]^{\circ}\ar[dd]_{\pi_{2}} &  & \mathcal{C}_{1}\ar[dd]^{s}\\
\\
 &  & \mathcal{C}_{0} &  &  & \mathcal{C}_{0} & \mathcal{C}_{1}\ar[rr]_{s} &  & \mathcal{C}_{0} & \mathcal{C}_{1}\ar[rr]_{s} &  & \mathcal{C}_{0}
}
\]
and
\[
\xymatrix@=1.5em{\mathcal{C}_{0}\times_{\mathcal{C}_{0}}\mathcal{C}_{1}\ar[r]^{i\times1}\ar[rdd]_{\pi_{2}} & \mathcal{C}_{1}\times_{\mathcal{C}_{0}}\mathcal{C}_{1}\ar[dd]^{\circ} & \mathcal{C}_{1}\times_{\mathcal{C}_{0}}\mathcal{C}_{0}\ar[l]_{1\times i}\ar[ldd]^{\pi_{1}} &  & \mathcal{C}_{1}\times_{\mathcal{C}_{0}}\mathcal{C}_{1}\times_{\mathcal{C}_{0}}\mathcal{C}_{1}\ar[rr]^{\circ\times1}\ar[dd]_{1\times\circ} &  & \mathcal{C}_{1}\times_{\mathcal{C}_{0}}\mathcal{C}_{1}\ar[dd]^{\circ}\\
\\
 & \mathcal{C}_{1} &  &  & \mathcal{C}_{1}\times_{\mathcal{C}_{0}}\mathcal{C}_{1}\ar[rr]_{\circ} &  & \mathcal{C}_{1}
}
\]
We omit the definitions of the relevant pullbacks.
\end{defn}

\begin{rem}
Categories internal to $\mathcal{E}$ embed into the category of simplicial
objects of $\mathcal{E}$ denoted $\left[\Delta^{\textnormal{op}},\mathcal{E}\right]$.
In the case $\mathcal{E}=\mathbf{Set}$ this is the usual nerve functor
\[
\xymatrix@=1.5em{\mathbf{Cat}\ar[rr] &  & \left[\Delta^{\textnormal{op}},\mathbf{Set}\right]\\
 & \Delta\ar[ur]\ar[lu]
}
\]
\end{rem}

\begin{defn}
A \emph{double category} is a category internal to $\mathbf{Cat}$.
In more detail, this consists of a set of objects, a category of horizontal
morphisms betwewen these objects, a category of vertical morphisms
between these objects, and squares of horizontal and vertical morphisms
\[
\xymatrix@=1.5em{X\ar@{~>}[dd]_{f}\ar[rr]^{a} &  & X'\ar@{~>}[dd]^{g}\\
\\
Y\ar[rr]_{b}\urtwocell{rruu}{\phi} &  & Y'
}
\]
which may be composed vertically and horizontally (satisfying compositional
axioms). In terms of the above definition, $\mathcal{C}_{0}$ is the
category of objects and horizontal morphisms, and $\mathcal{C}_{1}$
is the category of vertical morphisms and squares.
\end{defn}

\begin{rem}
One may iterate this process, taking a triple category to be a category
internal to $\mathbf{Dbl}$ (the category of double categories) and
so on. By iterating this process $n$ times one recovers the notion
of $n$-fold categories.
\end{rem}

One class of examples of double categories, is that for a monad $P$
on a category $\mathcal{C}$ we may form the Kleisli double category
of $P$, consisting of the two types of maps: arrows of $\mathcal{C}$
and $P$-Kleisli arrows of $\mathcal{C}$.
\begin{defn}
The \emph{Kleisli double category }of a monad $P$ on a category $\mathcal{C}$
consists of the objects of $\mathcal{C}$, has horizontal arrows those
of $\mathcal{C}$ and vertical morphisms $X\rightsquigarrow Y$ are
morphisms $X\to PY$ in $\mathcal{C}$ (called Kleisli arrows). A
square $\phi$ as on the left below

\[
\xymatrix@=1.5em{X\ar@{~>}[dd]_{f}\ar[rr]^{a} &  & X'\ar@{~>}[dd]^{g} &  &  &  & X\ar[dd]_{f}\ar[rr]^{a} &  & X'\ar[dd]^{g}\\
\\
Y\ar[rr]_{b}\urtwocell{rruu}{\phi} &  & Y' &  &  &  & PY\ar[rr]_{Pb} &  & PY'
}
\]
is the condition that the square on the right commutes.
\end{defn}

\begin{rem}
One may switch the horizontal and vertical morphisms of a double category,
giving the \emph{transpose }of a double category.
\end{rem}

\begin{rem}
As noted by Miranda \cite{Miranda}, the (transpose) Kleisli double
category of a monad $P$ can be seen as the nerve 2-functor
\[
\xymatrix@=1.5em{\mathbf{Kl}\left(\mathbf{Cat}\right)\ar[rr]^{N} &  & \left[\Delta^{\textnormal{op}},\mathbf{Cat}\right]\\
 & \Delta\ar[ur]\ar[lu]
}
\]
which is 2-fully faithful. The nerve here factors through double categories
and sends a monad $P$ to the double category with object of objects
$\mathcal{C}_{P\textnormal{-kl}}$ and object of morphisms $\left[2,\mathcal{C}\right]_{\left[2,P\right]\textnormal{-kl}}$
where 'kl' denotes the Kleisli category of a given monad $P$. Taking
the transpose of the nerve (equivalent to switching vertical and horizontal
transformations between double functors) we recover the embedding
$N^{T}\colon\mathbf{Mnd}\left(\mathbf{Cat}\right)\to\mathbf{Dbl}$.
\end{rem}

\begin{example}
As the nerve functor $N\colon\mathbf{Kl}\left(\mathbf{Cat}\right)\to\mathbf{Dbl}$
is fully faithful, and we have a Kleisli adjunction for a monad $P$
(now between the double category of squares and Klielsi double category),
we may consider the corresponding adjunction in the preimage. This
is given by the adjunction of monad morphisms
\[
\xymatrix@=1.5em{\mathcal{C}\ar[dd]_{1}\ar[rr]^{1} &  & \mathcal{C}\ar[dd]^{P} &  & \mathcal{C}\ar[dd]_{P}\ar[rr]^{P} &  & \mathcal{C}\ar[dd]^{1} &  & \mathcal{C}\ar[dd] &  & \mathcal{C}_{P}\ar[dd]\\
 &  &  & \dashv &  &  &  & \mapsto &  & \dashv\\
\mathcal{C}\ar[rr]_{1}\urtwocell{urur}{\eta} &  & \mathcal{C} &  & \mathcal{C}\ar[rr]_{P}\urtwocell{urur}{\mu} &  & \mathcal{C} &  & \mathcal{C}_{P} &  & \mathcal{C}
}
\]
This pair of monad morphisms allows one to go back and forth between
identity monads, and will be useful later on.\footnote{In the more general situation later on involving modified Kleisli
completions, we still have a $\epsilon\colon P\to1$, but no $\eta\colon1\to P$
and thus no adjunction of monad morphisms.}
\end{example}

\section{General Nerves}

In this section we consider more general examples of fully faithful
nerves of monads, aside from Kleisli double categories. The simple
idea is that instead of the 2-category $\mathbf{Kl}\left(\mathbf{Cat}\right)$,
we use a 2-category $\mathbf{mKl}\left(\mathbf{Cat}\right)$ which
is similar to $\mathbf{Kl}\left(\mathbf{Cat}\right)$ in that one
has a comparison 2-functor $\Phi\colon\mathbf{mKl}\left(\mathbf{Cat}\right)\to\mathbf{Kl}\left(\mathbf{Cat}\right)$
satisfying certain axioms. The notation 'mkl' is to denote a modified
version of the Kleisli completion. Moreover, to be more general and
also allow for iteration we will replace $\mathbf{Cat}$ by a more
general 2-category $\mathcal{K}$ satisfying some properties.

The following describes a set of axioms which will suffice to construct
such a fully faithful nerve functor on 2-categories of monads. Whilst
it may seem a lot of axioms, the majority of the conditions are trivial.
The central axiom is the existence of squares 
\[
\xymatrix@=1.5em{X\ar@{~>}[dd]_{f}\ar[rr]^{\Phi f} &  & PY\ar@{~>}[dd]^{\epsilon_{Y}}\\
\\
Y\ar@{=}[rr] &  & Y
}
\]
in the below, and this is what one checks when looking for examples.
It would not be surprising if some of the other axioms turn out to
be redundant.
\begin{thm}
Suppose we are given a 2-category $\mathcal{K}$, a 2-category $\mathbf{mKl}\left(\mathcal{K}\right)$,
a 2-functor $\Delta\to\mathcal{K}$ and an identity on objects, identity
on 1-cells 2-functor
\[
\Phi\colon\mathbf{mKl}\left(\mathcal{K}\right)\to\mathbf{Kl}\left(\mathcal{K}\right)
\]
with nerve functor denoted $N\colon\mathbf{mKl}\left(\mathcal{K}\right)\to\left[\Delta^{\textnormal{op}},\mathbf{Cat}\right]$.
Suppose that:
\begin{enumerate}
\item the 2-functor $\Delta\to\mathcal{K}$ is nice in that:
\begin{enumerate}
\item the nerve
\[
n\colon\mathcal{K}\to\left[\Delta^{\textnormal{op}},\mathbf{Cat}\right]
\]
factors through $\mathbf{Dbl}$ and is fully faithful; 
\item there exists double functors natural in $\mathcal{C}$ of the form\footnote{In $\mathbf{Cat}$ 2-cells between 1-cells $\mathbf{1}\to\mathcal{C}$
are the same as arrows (that is maps $\text{\ensuremath{\mathbf{2}}}\to\mathcal{C}$).
However, for a general $\mathcal{K}$ this may not be true, hence
the need for this axiom. The notation $\left(\textnormal{ob }\mathcal{K}\left(1,\mathcal{C}\right):\textnormal{ob }\mathcal{K}\left(2,\mathcal{C}\right)\right)$
refers to the category with objects $\textnormal{ob }\mathcal{K}\left(1,\mathcal{C}\right)$
and morphisms $\textnormal{ob }\mathcal{K}\left(2,\mathcal{C}\right)$.}
\[
\mathbf{Sq}\;\mathcal{K}\left(1,\mathcal{C}\right)\to n\mathcal{C}\to\mathbf{Sq}\left(\textnormal{ob }\mathcal{K}\left(1,\mathcal{C}\right):\textnormal{ob }\mathcal{K}\left(2,\mathcal{C}\right)\right)
\]
\end{enumerate}
\item the 2-functor $\Phi\colon\mathbf{mKl}\left(\mathcal{K}\right)\to\mathbf{Kl}\left(\mathcal{K}\right)$
is nice in that:
\begin{enumerate}
\item for all monads $P$ there exists a 2-cell $\epsilon\colon P\rightsquigarrow\textnormal{id}_{\mathcal{C}}$
in $\mathbf{mKl}\left(\mathcal{K}\right)$ such that $\Phi\epsilon\colon P\nrightarrow\textnormal{id}_{\mathcal{C}}$
in $\mathbf{Kl}\left(\mathcal{K}\right)$ is the identity on $P$;
\item for all monads $P$ on $\mathcal{C}$ and $f\colon X\rightsquigarrow Y\colon\mathbf{1}\to\left(P,\mathcal{C}\right)$,
there exists squares
\begin{equation}
\xymatrix@=1.5em{X\ar@{~>}[dd]_{f}\ar[rr]^{\Phi f} &  & PY\ar@{~>}[dd]^{\epsilon_{Y}}\\
\\
Y\ar@{=}[rr] &  & Y
}
\label{mainax}
\end{equation}
\end{enumerate}
\item the trivial axioms that:
\begin{enumerate}
\item for all squares
\[
\xymatrix@=1.5em{X\ar@{~>}[dd]_{f}\ar@{=}[rr] &  & X\ar@{~>}[dd]^{g}\\
\\
Y\ar@{=}[rr] &  & Y
}
\]
we have $f=g$;
\item every $\epsilon_{X}$ comprises a square
\[
\xymatrix@=1.5em{PX\ar@{~>}[dd]_{\epsilon_{X}}\ar[rr]^{\eta_{PX}} &  & P^{2}X\ar@{~>}[dd]^{P\epsilon_{X}}\\
\\
X\ar[rr]_{\eta_{X}} &  & PX
}
\]
\item the 2-functor
\[
\mathbf{mKl}\left(\mathcal{K}\right)\left[\mathbf{2},-\right]\colon\mathbf{mKl}\left(\mathcal{K}\right)\to\mathbf{Cat}
\]
is locally fully faithful;\footnote{This axiom will hold in most examples where the 2-cells of $\mathbf{mKl}\left(\mathcal{K}\right)$
have a natural definition.}
\item the nerve 2-functor $N\colon\mathbf{mKl}\left(\mathcal{K}\right)\to\mathbf{Dbl}$
is fully faithful on underlying 1-cells\footnote{These are the monad morphisms out of identity monads, which appear
in the underlying-inclusion adjunction. This is easy to check as these
are all trivial monad morphisms.}
\end{enumerate}
\end{enumerate}
Then the nerve $N\colon\mathbf{mKl}\left(\mathcal{K}\right)\to\mathbf{Dbl}$
is $\epsilon$-fully faithful. That is, full on double functors $F\colon\mathscr{C}_{P\textnormal{-mkl}}\to\mathscr{D}_{Q\textnormal{-mkl}}$
which are determined by their action on the class of morphisms $\left\{ PX\overset{\epsilon_{X}}{\rightsquigarrow}X\overset{f}{\rightsquigarrow}Y\right\} \cup\left\{ PX\overset{Pf}{\rightsquigarrow}PY\right\} $\footnote{The class of morphisms including those of the form $PX\overset{\epsilon_{X}}{\rightsquigarrow}X\overset{f}{\rightsquigarrow}Y$
and $PX\overset{Pf}{\rightsquigarrow}PY$ is quite large, and so it
is expected that any reasonable double functor will be determined
by its action on this class. It is likely that all such double functors
are determined on this class, in which case $\epsilon$-fully faithfulness
is the same as fully faithfulness (though we have no proof of this).
Any counterexample would need to involve a very unusual double functor.}.
\end{thm}

\begin{proof}
We first consider the nerve $N\colon\mathbf{mKl}\left(\mathbf{Cat}\right)\to\mathbf{Dbl}$
on 1-cells. This is defined by the assignation

\[
\xymatrix@=1.5em{P\ar[dd]_{\left(F,\xi\right)} &  & \mathscr{C}_{P\textnormal{-stkl}}\ar[dd]_{F_{dbl}} &  & X\ar@{~>}[rr]^{f} &  & Y\\
 & \mapsto &  & :\\
Q &  & \mathscr{D}_{Q\textnormal{-stkl}} &  & FX\ar@{~>}[rr]^{\left(F,\xi\right)f} &  & FY
}
\]
that is whiskering with $\left(F,\xi\right)$, and this has the candidate
inverse assignation

\[
\xymatrix@=1.5em{\mathscr{C}_{P\textnormal{-stkl}}\ar[dd]_{F_{dbl}} &  & \mathscr{C}\ar[dd]_{F} & FPX\ar[dd]^{\Phi F_{dbl}\epsilon_{X}}\\
 & \mapsto\\
\mathscr{D}_{Q\textnormal{-stkl}} &  & \mathscr{D} & QFX
}
\]
We must check that $\xi:=\Phi F_{dbl}\epsilon$ is a well defined
2-cell $FP\Rightarrow QF$ in $\mathcal{K}$. To do this, note we
have the double natural transformation
\[
\xymatrix@=1.5em{\mathcal{C}_{1-mkl}\ar[r] & \mathcal{C}_{P-mkl}\ar@/^{1pc}/[rr]^{P}\ar@/_{1pc}/[rr]_{1} & \dtwocell{}{\epsilon} & \mathcal{C}_{P-mkl}\ar[r]^{F} & \mathcal{D}_{Q-mkl}}
\]
which consists of components
\[
\xymatrix@=1.5em{\mathbf{mKl}\left(2,1_{\mathcal{C}}\right)\ar@/^{1pc}/[rr]^{FP}\ar@/_{1pc}/[rr]_{F} & \dtwocell{}{F\epsilon} & \mathbf{mKl}\left(2,\mathcal{Q}\right) &  & \textnormal{in }\mathbf{Cat}}
\]
which corresponds to (using fullness on underlying monad morphisms
and local fully faithfulness at $\mathbf{2}$)
\[
\xymatrix@=1.5em{1_{\mathcal{C}}\ar@/^{1pc}/[rr]^{FP}\ar@/_{1pc}/[rr]_{F} & \dtwocell{}{F\epsilon} & \mathcal{Q} &  & \textnormal{in }\mathbf{\mathbf{mKl}\left(\mathcal{K}\right)}}
\]
giving by $\Phi$ the Kleisli map $FP\Rightarrow F$ i.e. 2-cell $FP\Rightarrow QF$.
The monad morphism axioms follow from applying $F$ and $\Phi$ to
\[
\xymatrix@=1.5em{X\ar[rr]^{\eta_{X}}\ar@{=}[dd] &  & PX\ar@{~>}[dd]^{\epsilon_{X}} &  &  & P^{2}X\ar[rr]^{\mu_{X}}\ar@{~>}[d]_{\epsilon_{PX}} &  & PX\ar@{~>}[dd]^{\epsilon_{X}}\\
 &  &  &  &  & PX\ar@{~>}[d]_{\epsilon_{X}}\\
X\ar@{=}[rr] &  & X &  &  & X\ar@{=}[rr] &  & X
}
\]
and using that $n$ is fully faithful. For faithfulness, we must show
$\left(F,\xi\right)=\left(F,\Phi F_{dbl}\epsilon\right)$ for any
$F_{dbl}$ which is equal to whiskering by some $\left(F,\xi\right)$.
Clearly the underlying $F$ is sent to $F$, moreover for all $X\in\mathscr{C}$
\[
\Phi F_{dbl}\epsilon_{X}=\Phi\left[\left(F,\xi\right)\epsilon_{X}\right]=\Phi\left(F,\xi\right)\Phi\epsilon_{X}=\left(F,\xi\right)\textnormal{id}_{X}=\xi_{X}
\]
We now check fullness, and suppose we are given a double functor $F_{dbl}\colon\mathscr{C}_{P\textnormal{-mkl}}\to\mathscr{D}_{Q\textnormal{-mkl}}$
which is completely determined by its action on the class $\left\{ PX\overset{\epsilon_{X}}{\rightsquigarrow}X\overset{f}{\rightsquigarrow}Y\right\} \cup\left\{ PX\overset{Pf}{\rightsquigarrow}PY\right\} $.
To show this double functor lies in the image of the nerve, we must
check that $F_{dbl}$ is equal to the operation of whiskering by $\left(F,\xi\right)$
(with $\xi\colon=\Phi F_{dbl}\epsilon_{X}$). We need only show this
for compisites $PX\overset{\epsilon_{X}}{\rightsquigarrow}X\overset{f}{\rightsquigarrow}Y$.
Firstly, we see that $F_{dbl}\epsilon_{X}=\left(F,\xi\right)\epsilon_{X}$
from the diagram 
\[
\xymatrix@=1.5em{FPX\ar[dd]_{F_{\textnormal{dbl}}\epsilon_{X}}\ar[rr]^{F\eta PX} &  & FP^{2}X\ar[dd]^{FP\epsilon=\left(F,\xi\right)P\epsilon}\ar[rr]^{F\mu X} &  & FPX\ar[dddd]^{\left(F,\xi\right)\epsilon_{X}}\\
\\
FX\ar[rr]_{F\eta X} &  & FPX\ar[dd]^{\left(F,\xi\right)\epsilon}\\
\\
FX\ar@{=}[rr]\ar@{=}[uu] &  & FX\ar@{=}[rr] &  & FX
}
\]
and the fact that $FP$ lies in the image, because it factors through
an underlying as 
\begin{equation}
\xymatrix@=1.5em{\mathscr{C}_{P\textnormal{-mkl}}\ar[r] & \mathscr{C}_{1\textnormal{-mkl}}\ar[r] & \mathscr{C}_{P\textnormal{-mkl}}\ar[r]^{F} & \mathscr{D}_{Q\textnormal{-mkl}}}
\label{facund}
\end{equation}
Note also that by middle four interchange we have
\[
\xymatrix@=1.5em{PX\ar[d]^{Pf} &  & PX\ar[d]^{\epsilon}\\
PY\ar[d]^{\epsilon_{Y}} & = & X\ar[d]^{f}\\
Y &  & Y
}
\]
Now as $F_{dbl}$ is determined on $\epsilon$ (by the above argument)
as well as $Pf$ (by \ref{facund}), it is thus determined on composites
as on the right above. Hence any double functor which is determined
by its action on the class $\left\{ PX\overset{\epsilon_{X}}{\rightsquigarrow}X\overset{f}{\rightsquigarrow}Y\right\} $
lies in the image.

We now consider the 2-cell aspect of 
\[
N\colon\mathbf{mKl}\left(\mathcal{K}\right)\to\mathbf{Dbl}
\]
and show that it is bijective on 2-cells (locally fully faithful).
A 2-cell $\alpha$ in $\mathbf{mKl}\left(\mathcal{K}\right)$ is assigned
as below

\[
\xymatrix@=1.5em{P\ar@/_{1pc}/[dd]_{\left(F,\xi\right)}\ar@/^{1pc}/[dd]^{\left(F',\xi'\right)} &  &  & \mathscr{C}_{P\textnormal{-stkl}}\ar@/_{1pc}/[dd]_{F_{dbl}}\ar@/^{1pc}/[dd]^{F_{dbl}'}\\
\rtwocell{}{\alpha} &  & \mapsto & \rtwocell{}{\alpha_{dbl}}\\
Q &  &  & \mathscr{D}_{Q\textnormal{-stkl}}
}
\]
where $\alpha_{dbl}$ is comprised of (recall $\mathscr{C}_{P\textnormal{-stkl}}$
has object of objects $\mathbf{mKl}\left(\mathcal{K}\right)\left[\mathbf{1},\mathscr{C}\right]$
and object of morphisms $\mathbf{mKl}\left(\mathcal{K}\right)\left[\mathbf{2},\mathscr{C}\right]$)
the natural transformations 

\[
\xymatrix@=1.5em{\mathbf{mKl}\left(\mathcal{K}\right)\left[\mathbf{1},\mathscr{C}\right]\ar@/_{1pc}/[dd]_{\circ\left(F,\xi\right)}\ar@/^{1pc}/[dd]^{\circ\left(F',\xi'\right)} &  &  & \mathbf{mKl}\left(\mathcal{K}\right)\left[\mathbf{2},\mathscr{C}\right]\ar@/_{1pc}/[dd]_{\circ\left(F,\xi\right)}\ar@/^{1pc}/[dd]^{\circ\left(F',\xi'\right)}\\
\rtwocell{}{\circ\alpha} &  &  & \rtwocell{}{\circ\alpha}\\
\mathbf{mKl}\left(\mathcal{K}\right)\left[\mathbf{1},\mathscr{D}\right] &  &  & \mathbf{mKl}\left(\mathcal{K}\right)\left[\mathbf{2},\mathscr{D}\right]
}
\]
and note the first is determined by the second (as we can restrict
along $1\to2$). The nerve being assumed locally fully faithful at
$\mathbf{2}$ then gives the result.
\end{proof}
\begin{rem}
It is useful to first consider the restriction to the identity monads
\[
\xymatrix@=1.5em{\mathbf{m}\left(\mathcal{K}\right)\ar[rr]^{\overline{\Phi}}\ar[d] &  & \mathcal{K}\ar[d]\\
\mathbf{mKl}\left(\mathcal{K}\right)\ar[rr]_{\Phi} &  & \mathbf{Kl}\left(\mathcal{K}\right)
}
\]
in checking the axioms and looking for examples. In particular, one
first checks horizontal composition of 2-cells in the simpler 2-category
$\mathbf{m}\left(\mathcal{K}\right)$.
\end{rem}

\begin{rem}
An additional functoriality axiom one might assume is 
\[
\xymatrix@R=1em{X\ar@{~>}[dd]_{f}\ar[r]^{\Phi f} & PY\ar@{~>}[dd]^{\epsilon_{Y}}\ar[rr]^{P\Phi g=\Phi Pg} &  & P^{2}Z\ar@{~>}[dd]_{\epsilon_{PZ}}\ar[r]^{\mu_{Z}} & PZ\ar@{~>}[dddd]^{\epsilon_{Z}} &  & X\ar@{~>}[dddd]_{gf}\ar[r]^{\Phi gf} & PZ\ar@{~>}[dddd]^{\epsilon_{Z}}\\
\\
Y\ar@{=}[r]\ar@{~>}[dd]_{g} & Y\ar[rr]^{\Phi g} &  & PZ\ar@{~>}[dd]_{\epsilon_{Z}} &  & =\\
\\
Z\ar@{=}[rrr] &  &  & Z\ar@{=}[r] & Z &  & Z\ar@{=}[r] & Z
}
\]
though this does not appear to be a necessary condition.
\end{rem}

Finally, we mention the transpose of the nerve which arises from switching
the horizontal and vertical transformations between double functors.
Provided that this functor $N^{T}\colon\mathbf{Mnd}\left(\mathcal{K}\right)\to\mathbf{Dbl}$
is well defined, it will be fully faithful.
\begin{cor}
Suppose further that for all $\alpha\colon\left(F,\xi\right)\to\left(G,\xi'\right)\colon\mathscr{C}\to\mathscr{D}$
(a monad 2-cell), and $\rho\colon X\rightsquigarrow Y\colon\mathbf{1}\to\mathscr{C}$
there exists a family of squares 
\[
\xymatrix@=1.5em{FX\ar@{~>}[dd]_{\left(F,\xi\right)\rho}\ar[rr]^{\alpha_{X}} &  & GX\ar@{~>}[dd]^{\left(G,\xi'\right)\rho}\\
 & \alpha_{\rho}\\
FY\ar[rr]_{\alpha_{Y}} &  & GY
}
\]
which compose vertically and horizontally, then the transpose $N^{T}\colon\mathbf{Mnd}\left(\mathcal{K}\right)\to\mathbf{Dbl}$
is also $\epsilon$-fully faithful.\footnote{If $\mathcal{K}$ is not assumed to be $\mathbf{Cat}$, it can be
useful to upgrade $\mathbf{Dbl}$ to $\mathbf{Tpl}$ as the nerve
will forget some structure.}
\end{cor}

\section{Examples}

We now give some examples of various nerves on 2-categories of monads.
This means we must give examples of 2-categories $\mathbf{mKl}\left(\mathcal{K}\right)$
satisfying these given axioms. We will take $\mathcal{K}=\mathbf{Cat}$
here for simplicity. Of course the simplest example is to take $\mathbf{mKl}\left(\mathbf{Cat}\right)$
to be $\mathbf{Kl}\left(\mathbf{Cat}\right)$, in which case we have
the fully faithful nerve $N:\mathbf{Kl}\left(\mathbf{Cat}\right)\to\mathbf{Dbl}$
and its transpose $N^{T}:\mathbf{Mnd}\left(\mathbf{Cat}\right)\to\mathbf{Dbl}$
sending a monad $P$ to the Kleisli double category of $P$.

Recall that $\mathbf{mKl}\left(\mathbf{Cat}\right)$ has the same
objects and 1-cells as $\mathbf{Kl}\left(\mathbf{Cat}\right)$, thus
we need only specify the 2-cells of $\mathbf{mKl}\left(\mathbf{Cat}\right)$.
Typically in defining this 2-category, horizontal composition of 2-cells
is the most difficult calculation to verify, though it is true in
each of our examples satisfying axiom \eqref{mainax}. It is left
as an open question if condition \eqref{mainax} yields closure of
the 2-cells under horizontal compositon in general.

\subsection{Double categories of $P$-embeddings}

Let us now consider the construction that turns up in the setting
of AWFS \cite{AWFS2}. Here we take a 2-cell of $\mathbf{mKl}\left(\mathbf{Cat}\right)$
of the form $f\colon X\rightsquigarrow Y\colon\left(T,\mathcal{A}\right)\to\left(P,\mathcal{C}\right)$
to be a pair of maps $\pi\colon Y\to X$ and $\tau\colon X\to PY$
composing to the unit $\eta_{Y}$ (such pairs factoring a unit map
are also called $P$-split monos). The Kleisli arrow component $\tau$
must satisfy condition \eqref{klcond}. 
\begin{rem}
Another presentation of this data (which appears in \cite{WalkerDL})
is as a pair of maps $L\colon Y\to X$ and $\textnormal{res}_{L}\colon PX\to PY$
where $\textnormal{res}_{L}\cdot PL=\textnormal{id}$ and $\mathbf{res}_{L}$
is a $P$-homomorphism. This corresponds the choice of the free algebra
presentation or Kleisli arrow presentation of Kleisli categories.
It is clear from this presentation we have closure under vertical
composition of 2-cells.
\end{rem}

We leave it to the reader to check that the 2-cells are also closed
under horizontal composition and whiskering by monad morphisms. The
map $\Phi\colon\mathbf{mKl}\left(\mathbf{Cat}\right)\to\mathbf{Kl}\left(\mathbf{Cat}\right)$
simply forgets the extra structure sending $\left(\pi,\tau\right)$
to $\tau$, and each $\epsilon_{Y}$ is the factorization of the unit
as $\textnormal{id}_{PY}\cdot\eta_{Y}\colon Y\to PY\to PY$.

A general square in the double category $N\left(P\right)$ denoted
as on the left below, is such that both squares on the right below
commute
\[
\xymatrix@=1.5em{X\ar@{~>}[dd]_{\left(\pi,\tau\right)}\ar[rr]^{f} &  & X'\ar@{~>}[dd]^{\left(\pi',\tau'\right)} &  &  & X\ar@{~>}[dd]_{\tau}\ar[rr]^{f} &  & X'\ar@{~>}[dd]^{\tau'} & X\ar[rr]^{f} &  & X'\\
 &  &  &  & \colon\\
Y\ar[rr]_{g} &  & Y' &  &  & PY\ar[rr]_{Pg} &  & PY' & Y\ar[rr]_{g}\ar[uu]^{\pi} &  & Y'\ar[uu]_{\pi'}
}
\]
and clearly we have a square of the form \eqref{mainax} for any $\left(\pi,\tau\right)$.
The other axioms are all trivial conditions with any reasonable example,
and so we omit their verification. It follows that we have a fully
faithful nerve and nerve-transpose $N\colon\mathbf{mKl}\left(\mathbf{Cat}\right)\to\mathbf{Dbl}$
and $N^{T}\colon\mathbf{Mnd}\left(\mathbf{Cat}\right)\to\mathbf{Dbl}$.
\begin{rem}
One possible generalization is to take a 2-cell the be the data $\pi_{i}\colon Y\to X$
for $1\leq i\leq n$ and $\tau\colon X\to PY$ where every $\tau\cdot\pi_{i}$
is equal to the unit $\eta_{Y}$. We mention this generalization only
to point out there are infinitely many examples of fully faithful
nerves of monads.
\end{rem}

\subsection{Double categories of (strong) Kleisli split-epis}

In looking for examples of suitable 2-categories $\mathbf{mKl}\left(\mathbf{Cat}\right)$,
one thing to consider is that 2-cells should be closed under horizontal
composition. One example of such 2-cells are split epimorphisms, and
this motivates the following. Take a 2-cell of $\mathbf{mKl}\left(\mathbf{Cat}\right)$
of the form $f\colon X\rightsquigarrow Y\colon\left(T,\mathcal{A}\right)\to\left(P,\mathcal{C}\right)$
to be pairs of maps $s\colon Y\to PX$ and $e\colon X\to PY$ rendering
commutative

\[
\xymatrix@=1.5em{Y\ar[r]^{s}\ar[rd]_{\eta_{Y}} & PX\ar[r]^{Pe} & P^{2}Y\\
 & PY\ar[ur]_{\eta_{PY}}
}
\]
Note that such arrows are split epimorphisms in the sense of Kleisli
composition, seen by composing with the multiplication $\mu_{Y}\colon P^{2}Y\to PY$.
A general square in the double category $N\left(P\right)$ denoted
as on the left below, is data such that both squares on the right
below commute
\[
\xymatrix@=1.5em{X\ar@{~>}[dd]_{\left(\pi,\tau\right)}\ar[rr]^{f} &  & X'\ar@{~>}[dd]^{\left(\pi',\tau'\right)} &  &  & X\ar@{~>}[dd]_{e}\ar[rr]^{f} &  & X'\ar@{~>}[dd]^{e'} & PX\ar[rr]^{Pf} &  & PX'\\
 &  &  &  & \colon\\
Y\ar[rr]_{g} &  & Y' &  &  & PY\ar[rr]_{Pg} &  & PY' & Y\ar[rr]_{g}\ar[uu]^{s} &  & Y'\ar[uu]_{s'}
}
\]
It is an easy exercise to check this notion of 2-cell is closed under
vertical composition. A more interesting exercise is to check that
this definition of 2-cell is closed under horizontal composition and
whiskering by monad morphisms, a fact which is much less obvious (though
simple and tedious to check). Again this yields fully faithful nerve
functors $N\colon\mathbf{mKl}\left(\mathbf{Cat}\right)\to\mathbf{Dbl}$
and $N^{T}\colon\mathbf{Mnd}\left(\mathbf{Cat}\right)\to\mathbf{Dbl}$.

\section{Iteration}

Given a pair of examples of fully faithful nerve functors $N_{1}^{T}\colon\mathbf{Mnd}\left(\mathbf{Cat}\right)\to\mathbf{Dbl}$
and $N_{2}^{T}\colon\mathbf{Mnd}\left(\mathbf{Cat}\right)\to\mathbf{Dbl}$,
we may iterate to construct a fully faithful functor from $\mathbf{Mnd}\left(\mathbf{Mnd}\left(\mathbf{Cat}\right)\right)\to\mathbf{Tpl}$
sending distributive laws to triple categories, and more generally
functors $\mathbf{Mnd}^{n}\left(\mathbf{Cat}\right)\to\left(n+1\right)\mathbf{fold}$
sending $n$-ary distributive laws to $\left(n+1\right)$fold categories.
This works by first noting that when $\mathcal{K}=\mathbf{Dbl}$,
a nerve of the form $\mathbf{Mnd}\left(\mathbf{Dbl}\right)\to\mathbf{Dbl}$
factors through $\mathbf{Tpl}$, as both the object of objects and
object of morphisms of this double category, are themselves double
categories. Thus we have the composite
\[
\xymatrix@=1.5em{\mathbf{Mnd}\left(\mathbf{Mnd}\left(\mathbf{Cat}\right)\right)\myar{\mathbf{Mnd}\left(N_{1}^{T}\right)}{rrr} &  &  & \mathbf{Mnd}\left(\mathbf{Dbl}\right)\ar[rrr]^{N_{2}^{T}} &  &  & \text{\ensuremath{\mathbf{Tpl}}}}
\]
Iterating this process yields embeddings $\mathbf{Mnd}^{n}\left(\mathbf{Cat}\right)\to\left(n+1\right)\mathbf{fold}$,
and one can mix any family of examples of nerves of monads in any
order. 
\begin{example}
[The triple category of embeddings of a distributive law] Given two
copies of the nerve functor sending a monad $P$ to the double category
of $P$-embeddings $\mathscr{C}_{P\textnormal{-emb}}$, we have an
embedding of distributive laws into triple categories. We now work
out what this triple category is. 

Suppose we are given monads $T$ and $P$, and a distributive law
$\lambda\colon TP\to PT$. We have since $\mathbf{Mnd}\left(\mathbf{Cat}\right)\to\mathbf{Dbl}$
is fully faithful, a double category $\mathscr{C}_{P\textnormal{-emb}}$
and a double monad $\widetilde{T}$ on it. Given the 1-category $\mathscr{C}_{P\textnormal{-emb}}^{1}$
(using ``1'' to refer to the object of morphisms and ``0'' the
object of objects), we get from the monad $\widetilde{T}_{1}$ a double
category $\mathscr{C}_{\widetilde{T}_{1}\textnormal{-emb}}$. This
gives a diagram
\[
\xymatrix{\mathscr{C}_{\widetilde{T}_{1}\textnormal{-emb}}^{1}\ar@<-1ex>[d]\ar@<1ex>[d]\ar@<-1ex>[r]\ar@<1ex>[r] & \mathscr{C}_{P\textnormal{-emb}}^{1}\ar[l]\ar@<-1ex>[d]\ar@<1ex>[d]\\
\mathscr{C}_{\widetilde{T}_{1}\textnormal{-emb}}^{0}\ar[u]\ar@<-1ex>[r]\ar@<1ex>[r] & \mathscr{C}_{P\textnormal{-emb}}^{0}\ar[l]\ar[u]
}
\]
as a double-category internal to $\mathbf{Cat}$, i.e. a triple category.
This may also be written 
\[
\xymatrix{\mathscr{C}_{\widetilde{T}_{1}\textnormal{-emb}}^{1}\ar@<-1ex>[d]\ar@<1ex>[d]\ar@<-1ex>[r]\ar@<1ex>[r] & \mathscr{C}_{P\textnormal{-emb}}^{1}\ar[l]\ar@<-1ex>[d]\ar@<1ex>[d]\\
\mathscr{C}_{T\textnormal{-emb}}^{1}\ar[u]\ar@<-1ex>[r]\ar@<1ex>[r] & \mathscr{C}\ar[l]\ar[u]
}
\]
To make this clear, we write down what general objects in these categories
are (using ``res'' notation for $P$-embeddings and ``unit-factorization''
notation for $T$-embeddings to easily distinguish them)
\[
\xymatrix@R=1em{\text{\ensuremath{\left(L,\text{\textnormal{res}}_{L}\right)}}\to\text{\ensuremath{\left(K,\text{\textnormal{res}}_{K}\right)}\ensuremath{\ensuremath{\to T_{1}\text{\ensuremath{\left(L,\text{\textnormal{res}}_{L}\right)}}}}} & \text{\ensuremath{\left(L,\text{\textnormal{res}}_{L}\right)}}\\
X\to Y\to TX & X
}
\]
Thus a distributive law $\lambda\colon TP\to PT$ allows us to assemble
together the double category of $P$-embeddings and double category
of $T$-embeddings into a single triple category.\footnote{In the talk \cite{PareKlTalk} a related construction with distributive
laws was mentioned, with two Kleisli double categories being combined
with a distributive law. The iterated version on the simpler Kleisli
case was also seen by Miranda \cite{Miranda}. }
\end{example}

\section{Future work}

The original motivation of this work was to construct various examples
of embeddings out of 2-categories of monads $N\colon\mathbf{Mnd}\left(\mathbf{Cat}\right)\to\mathbf{Dbl}$
to prove results about distributive laws. As such an embedding is
fully faithful, a distributive law may be identified with a double
monad on a double category $N\left(P\right)$. This approach is most
useful in reducing the coherence axioms for pseudodistributive laws
involving KZ pseudomonads \cite{kock1972} (using the ``$P$-embedding''
nerve which shows up in AWFS \cite{AWFS2}). However, this would require
a generalization of these nerves to the setting of pseudomonads.

\bibliographystyle{plain}
\bibliography{references}

\end{document}